\documentclass[12pt, reqno]{amsart}
\usepackage{amsmath, amsthm, amscd, amsfonts, amssymb, graphicx, color}
\usepackage[bookmarksnumbered, colorlinks, plainpages]{hyperref}
\hypersetup{colorlinks=true,linkcolor=red, anchorcolor=green,
citecolor=cyan, urlcolor=red, filecolor=magenta, pdftoolbar=true}

\newtheorem{theorem}{Theorem}[section]
\newtheorem{lemma}[theorem]{Lemma}

\newtheorem{corollary}[theorem]{Corollary}
\theoremstyle{definition}

\theoremstyle{remark}
\newtheorem{remark}[theorem]{Remark}
\numberwithin{equation}{section}

\begin{document}

\setcounter{page}{1}

\title[Bounds for the $\mathbb{A}$-numerical radius of $2\times 2$ block matrices]
{Some upper bounds for the $\mathbb{A}$-numerical radius of $2\times 2$ block matrices}

\author[Q. Xu, Z. Ye \MakeLowercase{and} A. Zamani]
{Qingxiang Xu$^1$, Zhongming Ye$^1$ \MakeLowercase{and} Ali Zamani$^{2,*}$}

\address{$^1$Department of Mathematics, Shanghai Normal University, Shanghai 200234, P.R. China}
\email{qingxiang\_xu@126.com}

\address{$^1$Department of Mathematics, Shanghai Normal University, Shanghai 200234, P.R. China}
\email{zhongming\_ye@139.com}

\address{$^*$ Corresponding author, $^2$Department of Mathematics, Farhangian University, Tehran, Iran}
\email{zamani.ali85@yahoo.com}



\begin{abstract}
Let $\mathbb{A}=\left(
\begin{array}{cc}
A & 0 \\
0 & A \\
\end{array}
\right)$ be the $2\times2$ diagonal operator matrix determined by a positive bounded operator $A$.
For semi-Hilbertian operators $X$ and $Y$, we first show that
\begingroup\makeatletter\def\f@size{6}\check@mathfonts
\begin{align*}
w^2_{\mathbb{A}}\left(\begin{bmatrix}
0 & X \\
Y & 0
\end{bmatrix}\right) &\leq \frac{1}{4}\max\Big\{{\big\|XX^{\sharp_A} + Y^{\sharp_A}Y\big\|}_{A}, {\big\|X^{\sharp_A}X + YY^{\sharp_A}\big\|}_{A}\Big\}
+ \frac{1}{2}\max\big\{w_{A}(XY), w_{A}(YX)\big\},
\end{align*}
\endgroup
where $w_{\mathbb{A}}(\cdot)$, ${\|\cdot\|}_{A}$ and $w_{A}(\cdot)$
are the $\mathbb{A}$-numerical radius, $A$-operator seminorm and $A$-numerical radius, respectively.
We then apply the above inequality to find some upper bounds for the $\mathbb{A}$-numerical
radius of certain $2\times 2$ operator matrices.
In particular, we obtain some refinements of earlier $A$-numerical radius inequalities
for semi-Hilbertian operators.
An upper bound for the $\mathbb{A}$-numerical radius of $2\times 2$
block matrices of semi-Hilbertian space operators is also given.

\end{abstract} \maketitle
\section{Introduction}
Let $\big(\mathcal{H}, \langle\cdot, \cdot\rangle\big)$ be a complex Hilbert space equipped with the norm $\|\cdot\|$
and let $I$ stand for the identity operator on $\mathcal{H}$.
If $\mathcal{M}$ is a linear subspace of $\mathcal{H}$, then $\overline{\mathcal{M}}$ stands for its closure in the
norm topology of $\mathcal{H}$. We denote the orthogonal projection onto a closed linear subspace
$\mathcal{M}$ of $\mathcal{H}$ by $P_{\mathcal{M}}$.
Let $\mathbb{B}(\mathcal{H})$ be the algebra of all bounded linear operators on $\mathcal{H}$.
For every operator $T\in\mathbb{B}(\mathcal{H})$ its range is denoted by $\mathcal{R}(T)$,
its null space by $\mathcal{N}(T)$, and its adjoint by $T^*$.
Throughout this paper, we assume that $A\in\mathbb{B}(\mathcal{H})$ is a positive operator, which
induces a positive semidefinite sesquilinear
form ${\langle \cdot, \cdot\rangle}_A: \,\mathcal{H}\times \mathcal{H} \rightarrow \mathbb{C}$
defined by ${\langle x, y\rangle}_A = \langle Ax, y\rangle$.
We denote by ${\|\cdot\|}_A$ the seminorm induced by ${\langle \cdot, \cdot\rangle}_{A}$, that is,
${\|x\|}_A = \sqrt{{\langle x, x\rangle}_{A}}$
for every $x\in\mathcal{H}$.
Observe that ${\|x\|}_{A} = 0$ if and only if $x\in \mathcal{N}(A)$.
Then ${\|\cdot\|}_{A}$ is a norm if and only if $A$ is one-to-one, and the seminormed
space $(\mathcal{H}, {\|\cdot\|}_{A})$ is complete if and only if $\mathcal{R}(A)$ is closed in $\mathcal{H}$.
For $T\in \mathbb{B}(\mathcal{H})$, an operator $R\in \mathbb{B}(\mathcal{H})$
is called an $A$-adjoint operator of $T$ if for every $x, y\in \mathcal{H}$,
we have ${\langle Tx, y\rangle}_{A} = {\langle x, Ry\rangle}_{A}$, that is, $AR = T^*A$.
Generally, the existence of an $A$-adjoint operator is not guaranteed.
The set of all operators that admit $A$-adjoints is denoted by $\mathbb{B}_{A}(\mathcal{H})$.
By Douglas' Theorem \cite{Dou}, we have
\begin{align*}
\mathbb{B}_{A}(\mathcal{H}) = \big\{T\in\mathbb{B}(\mathcal{H}): \, \mathcal{R}(T^*A) \subseteq \mathcal{R}(A)\big\}.
\end{align*}
If $T\in\mathbb{B}_{A}(\mathcal{H})$, then the reduced solution to the equation $AX = T^*A$ is
denoted by $T^{\sharp_A}$, which is called the distinguished $A$-adjoint
operator of $T$. It verifies that $AT^{\sharp_A} = T^*A$, $\mathcal{R}(T^{\sharp_A}) \subseteq \overline{\mathcal{R}(A)}$ and $\mathcal{N}(T^{\sharp_A}) = \mathcal{N}(T^*A)$. Note that $T^{\sharp_A} = A^{\dag}T^*A$, where $A^{\dag}$ is the Moore--Penrose inverse of $A$, which
is the unique linear mapping from $\mathcal{R}(A) \oplus \mathcal{R}(A)^{\perp}$
into $\mathcal{H}$ satisfying the Moore--Penrose equations:
\begin{align*}
AXA = A, \,\, XAX = X, \,\, XA = P_{\overline{\mathcal{R}(A)}}
\,\, \mbox{and} \,\, AX = P_{\overline{\mathcal{R}(A)}}|_{\mathcal{R}(A) \oplus {\mathcal{R}(A)}^{\perp}}.
\end{align*}
In general, $A^{\dag} \not \in \mathbb{B}(\mathcal{H})$.
Indeed, $A^{\dag} \in \mathbb{B}(\mathcal{H})$ if and only if $A$ has closed range; see, for example, \cite{M.K.X}.
Notice that if $T\in\mathbb{B}_{A}(\mathcal{H})$, then $T^{\sharp_A}\in\mathbb{B}_{A}(\mathcal{H})$,
$(T^{\sharp_A})^{\sharp_A} = P_{\overline{\mathcal{R}(A)}}TP_{\overline{\mathcal{R}(A)}}$ and
$\big((T^{\sharp_A})^{\sharp_A}\big)^{\sharp_A} = T^{\sharp_A}$.
Also, for $T, S\in\mathbb{B}_{A}(\mathcal{H})$, it is easy to see that $TS\in\mathbb{B}_{A}(\mathcal{H})$
and $(TS)^{\sharp_A} = S^{\sharp_A}T^{\sharp_A}$.
An operator $U\in\mathbb{B}_{A}(\mathcal{H})$ is called $A$-unitary if
${\|Ux\|}_{A} = {\|U^{\sharp_A}x\|}_{A} = {\|x\|}_{A}$ for all $x\in\mathcal{H}$.
It should be mentioned  that, an operator $U\in\mathbb{B}_{A}(\mathcal{H})$ is $A$-unitary if and only if
$U^{\sharp_A}U = (U^{\sharp_A})^{\sharp_A}U^{\sharp_A} = P_{\overline{\mathcal{R}(A)}}$.
The set of all operators admitting $A^{1/2}$-adjoints is denoted by $\mathbb{B}_{A^{1/2}}(\mathcal{H})$.
Again, by applying Douglas' Theorem, we obtain
\begin{align*}
\mathbb{B}_{A^{1/2}}(\mathcal{H}) = \big\{T\in \mathbb{B}(\mathcal{H}): \,\, \exists c>0;
\,\,{\|Tx\|}_{A} \le  c{\|x\|}_{A}, \,\, \forall x\in \mathcal{H}\big\}.
\end{align*}
Clearly, ${\langle \cdot, \cdot\rangle}_{A}$ induces a seminorm on $\mathbb{B}_{A^{1/2}}(\mathcal{H})$.
Indeed, if $T\in\mathbb{B}_{A^{1/2}}(\mathcal{H})$, then it holds that
\begin{align*}
{\|T\|}_{A} := \displaystyle{\sup_{0\neq x\in\overline{\mathcal{R}(A)}}}\frac{{\|Tx\|}_{A}}{{\|x\|}_{A}}
=\sup\big\{{\|Tx\|}_{A}: \,\, x\in\mathcal{H}, {\|x\|}_A =1\big\} < +\infty.
\end{align*}
It can be verified that, for $T\in\mathbb{B}_{A^{1/2}}(\mathcal{H})$,
we have ${\|Tx\|}_{A}\leq {\|T\|}_{A}{\|x\|}_{A}$ for all $x\in \mathcal{H}$.
This implies that, for $T, S\in\mathbb{B}_{A^{1/2}}(\mathcal{H})$, we have ${\|TS\|}_{A}\leq {\|T\|}_{A}{\|S\|}_{A}$.
Notice that it may happen that ${\|T\|}_A = + \infty$
for some $T\in\mathbb{B}(\mathcal{H})\setminus \mathbb{B}_{A^{1/2}}(\mathcal{H})$.
For example, let $A$ be the diagonal operator on
the Hilbert space $\ell^2$ given by $Ae_n = \frac{e_n}{n!}$,
where $\{e_n\}$ denotes the canonical basis of $\ell^2$
and consider the left shift operator $T \in \mathbb{B}(\ell^2)$.
An operator $T\in \mathbb{B}(\mathcal{H})$ is called $A$-positive if $AT$ is positive.
Note that if $T$ is $A$-positive, then
${\|T\|}_A = \sup\big\{{\langle Tx, x\rangle}_A:\,\,\, x\in \mathcal{H},\, {\|x\|}_A = 1\big\}$.
An operator $T\in\mathbb{B}(\mathcal{H})$ is said to be $A$-selfadjoint if $AT$ is selfadjoint,
that is, $AT = T^*A$. Obviously, an $A$-positive operator is always an $A$-selfadjoint operator.
Observe that if $T$ is $A$-selfadjoint, then $T\in\mathbb{B}_{A}(\mathcal{H})$.
However, it does not hold, in general, that $T = T^{\sharp_A}$.
More precisely, if $T\in\mathbb{B}_{A}(\mathcal{H})$, then $T = T^{\sharp_A}$ if and only if
$T$ is $A$-selfadjoint and $\mathcal{R}(T) \subseteq \overline{\mathcal{R}(A)}$.
Note that for $T\in\mathbb{B}_{A}(\mathcal{H})$, $T^{\sharp_A}T$ and  $TT^{\sharp_A}$ are $A$-selfadjoint and so
\begin{align}\label{I.1.01}
{\|T^{\sharp_A}T\|}_A = {\|TT^{\sharp_A}\|}_A = {\|T\|}^2_A = {\|T^{\sharp_A}\|}^2_A.
\end{align}
For an account of results, we refer the reader to \cite{Ar.Co.Go, Ar.Co.Go.2, Ma.Se.Su, Su}.

The $A$-numerical radius of $T\in\mathbb{B}(\mathcal{H})$ (see \cite{Ba.Ka.Ah} and the references therein)
is defined by
\begin{align*}
w_{A}(T) = \sup\Big\{\big|{\langle Tx, x\rangle}_A\big|: \,\,\, x\in \mathcal{H},\, {\|x\|}_A = 1\Big\}.
\end{align*}
It was recently shown in \cite[Theorem 2.5]{Z.3} that if $T\in\mathbb{B}_{A}(\mathcal{H})$, then
\begin{align*}
w_{A}(T) = \displaystyle{\sup_{\theta \in \mathbb{R}}}
{\left\|\frac{e^{i\theta}T + (e^{i\theta}T)^{\sharp_A}}{2}\right\|}_A.
\end{align*}
Moreover, it is known that if $T$ is $A$-selfadjoint, then $w_A(T) = {\|T\|}_{A}$.
One of the important properties of $w_{A}(\cdot)$ is that it is weakly $A$-unitarily
invariant (see \cite{B.F.P}), that is,
\begin{align}\label{I.1.03}
w_{A}(U^{\sharp_{A}}TU)= w_{A}(T),
\end{align}
for every $A$-unitary $U\in\mathbb{B}_{A}(\mathcal{H})$.
Another basic fact about the $A$-numerical radius is the power
inequality (see \cite{M.X.Z}), which asserts that
\begin{align}\label{I.1.04}
w_{A}(T^n) \leq w^n_{A}(T) \qquad (n\in\mathbb{N}).
\end{align}
Notice that it may happen that $w_A(T) = + \infty$ for some $T\in\mathbb{B}(\mathcal{H})$.
Indeed, one can take
$A = \begin{bmatrix}
1 & 0 \\
0 & 0
\end{bmatrix}$ and
$T = \begin{bmatrix}
0 & 1 \\
1 & 0
\end{bmatrix}$.
In fact, if $T\in\mathbb{B}(\mathcal{H})$ is such that
$T\big(\mathcal{N}(A)\big)\nsubseteq \mathcal{N}(A)$, then $w_A(T) = + \infty$.
However, $w_{A}(\cdot)$ is a seminorm on $\mathbb{B}_{A^{1/2}}(\mathcal{H})$
which is equivalent to the $A$-operator seminorm ${\|\cdot\|}_{A}$.
Namely, for $T\in \mathbb{B}_{A^{1/2}}(\mathcal{H})$, we have
\begin{align}\label{I.1.02}
\frac{1}{2}{\|T\|}_{A} \leq w_{A}(T)\leq {\|T\|}_{A}.
\end{align}
For proofs and more facts about $A$-numerical radius of operators, we refer the reader to
\cite{Ba.Ka.Ah, Ba.Ka.Ah.2, B.F.P, B.P.N, Feki.1, Feki.Sid, M.X.Z, Z.3}.

Now, let $\mathbb{A}$ be the $2\times2$ diagonal operator matrix whose diagonal entries are $A$,
that is, $\mathbb{A} = \begin{bmatrix}
A & 0 \\
0 & A
\end{bmatrix}$.
Clearly, $\mathbb{A} \in \mathbb{B}(\mathcal{H}\oplus\mathcal{H})^{+}$
and so $\mathbb{A}$ induces the semi-inner product
${\langle x, y\rangle}_{\mathbb{A}} = \langle \mathbb{A}x, y\rangle = {\langle x_1, y_1\rangle}_{A} + {\langle x_2, y_2\rangle}_{A}$,
for all $x = (x_1, x_2)\in \mathcal{H}\oplus\mathcal{H}$ and $y = (y_1, y_2)\in \mathcal{H}\oplus\mathcal{H}$.
Very recently, inspired by the numerical radius equalities and inequalities for operator matrices in
\cite{H.K.S.1, H.K.S.2, M.S, S.D.M, Sh}, some inequalities for the $\mathbb{A}$-numerical radius of $2\times 2$ operator matrices
have been computed in \cite{B.F.P, B.P.N}.
This paper is also devoted  to the study of the $\mathbb{A}$-numerical radius of $2\times 2$ block matrices.
More precisely, we first derive an upper bound for the $\mathbb{A}$-numerical
radius of the off-diagonal operator matric $\begin{bmatrix}
0 & X \\
Y & 0
\end{bmatrix}$ defined on $\mathcal{H}\oplus\mathcal{H}$.
We then find some upper bounds for the $\mathbb{A}$-numerical
radius of certain $2\times 2$ block matrices.
In particular, we obtain a refinement on the second inequality (\ref{I.1.02}).
Finally, we compute a new upper bound for the $\mathbb{A}$-numerical
radius of $2\times 2$ operator matrices.
Some of the obtained results are new even in the case that the underlying operator $A$ is the identity operator.
\section{Results}
In order to achieve the goal of this section, we need the following lemmas.
The first lemma was recently given in \cite{B.F.P, B.P.N}.
\begin{lemma}\label{L.2.01}
Let $T, S, X, Y\in\mathbb{B}_{A^{1/2}}(\mathcal{H})$. Then
\begin{itemize}
\item[(i)] ${\begin{bmatrix}
T & X \\
Y & S
\end{bmatrix}}^{\sharp_{\mathbb{A}}} = \begin{bmatrix}
T^{\sharp_{A}} & Y^{\sharp_{A}} \\
X^{\sharp_{A}} & S^{\sharp_{A}}
\end{bmatrix}$.
\item[(ii)] ${\left\|\begin{bmatrix}
X & 0 \\
0 & Y
\end{bmatrix}\right\|}_{\mathbb{A}} = {\left\|\begin{bmatrix}
0 & X \\
Y & 0
\end{bmatrix}\right\|}_{\mathbb{A}} = \max\big\{{\|X\|}_{A}, {\|Y\|}_{A}\big\}$.
\item[(iii)] $w_{\mathbb{A}}\left(\begin{bmatrix}
X & 0 \\
0 & Y
\end{bmatrix}\right) = \max\big\{w_{A}(X), w_{A}(Y)\big\}$.
\item[(iv)] $w_{\mathbb{A}}\left(\begin{bmatrix}
X & Y \\
Y & X
\end{bmatrix}\right) = \max\big\{w_{A}(X+Y), w_{A}(X-Y)\big\}$.

In particular, $w_{\mathbb{A}}\left(\begin{bmatrix}
0 & Y \\
Y & 0
\end{bmatrix}\right) = w_{A}(Y)$.
\end{itemize}
\end{lemma}
The second lemma is stated as follows.
\begin{lemma}\label{L.2.02}
Let $x, y, z \in \mathcal{H}\oplus\mathcal{H}$ with ${\|z\|}_{\mathbb{A}} = 1$. Then
\begin{align*}
\big|{\langle x, z\rangle}_{\mathbb{A}}{\langle z, y\rangle}_{\mathbb{A}}\big|
\leq \frac{1}{2}\big({\|x\|}_{\mathbb{A}}{\|y\|}_{\mathbb{A}} + |{\langle x, y\rangle}_{\mathbb{A}}|\big).
\end{align*}
\end{lemma}
\begin{proof}
Notice first that, by \cite{Buz}, we have
\begin{align}\label{drago}
|\langle a, c\rangle\langle c, b\rangle| \leq \frac{1}{2}\big(\|a\|\,\|b\| + |\langle a, b\rangle|\big),
\end{align}
for every $a, b, c\in \mathcal{H}\oplus\mathcal{H}$ with $\|c\| = 1$.
Now, let $x, y, z \in \mathcal{H}\oplus\mathcal{H}$ with ${\|z\|}_{\mathbb{A}} = 1$.
Since $\|{\mathbb{A}}^{1/2}z\| = 1$, then by using (\ref{drago}), we see that
\begin{align*}
\big|{\langle x, z\rangle}_{\mathbb{A}}{\langle z, y\rangle}_{\mathbb{A}}\big|
& = \big|\langle {\mathbb{A}}^{1/2}x, {\mathbb{A}}^{1/2}z\rangle\langle {\mathbb{A}}^{1/2}z, {\mathbb{A}}^{1/2}y\rangle\big|
\\& \leq \frac{1}{2}\Big(\|{\mathbb{A}}^{1/2}x\|\,\| {\mathbb{A}}^{1/2}y\| + |\langle {\mathbb{A}}^{1/2}x,  {\mathbb{A}}^{1/2}y\rangle|\Big)
\\& = \frac{1}{2}\big({\|x\|}_{\mathbb{A}}{\|y\|}_{\mathbb{A}} + |{\langle x, y\rangle}_{\mathbb{A}}|\big).
\end{align*}
This proves the desired result.
\end{proof}
Now, we are in the position to state an upper bound for the $\mathbb{A}$-numerical
radius of the off-diagonal part $\begin{bmatrix}
0 & X \\
Y & 0
\end{bmatrix}$ of a $2\times 2$ operator matric $\begin{bmatrix}
T & X \\
Y & S
\end{bmatrix}$ defined on $\mathcal{H}\oplus\mathcal{H}$.
\begin{theorem}\label{T.2.03}
Let $X, Y\in\mathbb{B}_{A^{1/2}}(\mathcal{H})$. Then

\begingroup\makeatletter\def\f@size{9}\check@mathfonts
\begin{align*}
w^2_{\mathbb{A}}\left(\begin{bmatrix}
0 & X \\
Y & 0
\end{bmatrix}\right) &\leq \frac{1}{4}\max\Big\{{\big\|X^{\sharp_A}X + YY^{\sharp_A}\big\|}_{A}, {\big\|XX^{\sharp_A} + Y^{\sharp_A}Y\big\|}_{A}\Big\}
+ \frac{1}{2}\max\big\{w_{A}(XY), w_{A}(YX)\big\}.
\end{align*}
\endgroup
\end{theorem}
\begin{proof}
Put $M=\begin{bmatrix}
0 & X \\
Y & 0
\end{bmatrix}$, $N=\begin{bmatrix}
XX^{\sharp_{A}} & 0 \\
0 & YY^{\sharp_{A}}
\end{bmatrix}$, $P=\begin{bmatrix}
Y^{\sharp_{A}}Y & 0 \\
0 & X^{\sharp_{A}}X
\end{bmatrix}$ and $Q=\begin{bmatrix}
XY & 0 \\
0 & YX
\end{bmatrix}$, to simplify the writing.
Then
\begin{align}\label{I.1.T.2.03}
P+N = \begin{bmatrix}
XX^{\sharp_{A}} + Y^{\sharp_{A}}Y & 0 \\
0 & X^{\sharp_{A}}X + YY^{\sharp_{A}}
\end{bmatrix}.
\end{align}
Furthermore, by Lemma \ref{L.2.01}(i), we get
\begin{align}\label{I.2.T.2.03}
MM^{\sharp_{\mathbb{A}}} = N, \quad M^{\sharp_{\mathbb{A}}}M = P \quad \mbox{and} \quad M^2 = Q.
\end{align}
Now, let $z\in \mathcal{H}\oplus\mathcal{H}$ with ${\|z\|}_{\mathbb{A}} = 1$.
Since ${\langle Mz, z\rangle}_{\mathbb{A}} =
{\langle z,
M^{\sharp_{\mathbb{A}}}z\rangle}_{\mathbb{A}}$, we have

\begingroup\makeatletter\def\f@size{10}\check@mathfonts
\begin{align*}
2\big|{\langle Mz, z\rangle}_{\mathbb{A}}\big|^2& = 2\Big|{\langle Mz, z\rangle}_{\mathbb{A}}
{\langle z, M^{\sharp_{\mathbb{A}}}z\rangle}_{\mathbb{A}}\Big|
\\& \leq {\|Mz\|}_{\mathbb{A}}
{\|M^{\sharp_{\mathbb{A}}}z\|}_{\mathbb{A}}
+ \big|{\langle Mz, M^{\sharp_{\mathbb{A}}}z\rangle}_{\mathbb{A}}\big|
\qquad \qquad \qquad \quad \big(\mbox{by Lemma \ref{L.2.02}}\big)
\\& = \sqrt{{\langle Mz, Mz\rangle}_{\mathbb{A}}
{\langle M^{\sharp_{\mathbb{A}}}z, M^{\sharp_{\mathbb{A}}}z\rangle}_{\mathbb{A}}}
+ \big|{\langle Mz, M^{\sharp_{\mathbb{A}}}z\rangle}_{\mathbb{A}}\big|
\\& = \sqrt{{\langle M^{\sharp_{\mathbb{A}}}Mz, z\rangle}_{\mathbb{A}}
{\langle MM^{\sharp_{\mathbb{A}}}z, z\rangle}_{\mathbb{A}}}
+ \big|{\langle M^2z, z\rangle}_{\mathbb{A}}\big|
\\& = \sqrt{{\langle Pz, z\rangle}_{\mathbb{A}}
{\langle Nz, z\rangle}_{\mathbb{A}}} + |{\langle Qz, z\rangle}_{\mathbb{A}}|
\qquad \qquad \qquad \qquad \qquad \quad \big(\mbox{by (\ref{I.2.T.2.03})}\big)
\\& \leq \frac{1}{2}\Big({\langle Pz, z\rangle}_{\mathbb{A}} +
{\langle Nz, z\rangle}_{\mathbb{A}}\Big) + w_{\mathbb{A}}(Q)
\\& \qquad \qquad \qquad \qquad \big(\mbox{by the arithmetic-geometric mean inequality}\big)
\\& = \frac{1}{2}{\big\langle (P + N)z, z\big\rangle}_{\mathbb{A}} + \max\big\{w_{A}(XY), w_{A}(YX)\big\}
\qquad \big(\mbox{by Lemma \ref{L.2.01}(iii)}\big)
\\& \leq \frac{1}{2}{\|P+N\|}_{\mathbb{A}}
+ \max\big\{w_{A}(XY), w_{A}(YX)\big\}
\\& \qquad \qquad \qquad  \qquad \Big(\mbox{since $P+N$ is an $\mathbb{A}$-positive operator}\Big)
\\& = \frac{1}{2}\max\Big\{{\big\|X^{\sharp_A}X + YY^{\sharp_A}\big\|}_{A}, {\big\|XX^{\sharp_A} + Y^{\sharp_A}Y\big\|}_{A}\Big\}
+ \max\big\{w_{A}(XY), w_{A}(YX)\big\}.
\\& \qquad \qquad \qquad \qquad \qquad \qquad \qquad \qquad \big(\mbox{by Lemma \ref{L.2.01}(ii) and (\ref{I.1.T.2.03})}\big)
\end{align*}
\endgroup
Hence
\begingroup\makeatletter\def\f@size{8}\check@mathfonts
\begin{align*}
\big|{\langle Mz, z\rangle}_{\mathbb{A}}\big|^2& \leq \frac{1}{4}\max\Big\{{\big\|X^{\sharp_A}X + YY^{\sharp_A}\big\|}_{A}, {\big\|XX^{\sharp_A} + Y^{\sharp_A}Y\big\|}_{A}\Big\}
+ \frac{1}{2}\max\big\{w_{A}(XY), w_{A}(YX)\big\}.
\end{align*}
\endgroup
Taking the supremum in the above inequality over $z\in \mathcal{H}\oplus\mathcal{H}$ with ${\|z\|}_{\mathbb{A}} = 1$,
we deduce the desired inequality.
\end{proof}
\begin{remark}\label{R.2.04}
Let $X, Y\in\mathbb{B}_{A^{1/2}}(\mathcal{H})$.
By Theorem \ref{T.2.03}, (\ref{I.1.01}) and (\ref{I.1.02}), we have
\begingroup\makeatletter\def\f@size{9}\check@mathfonts
\begin{align*}
w_{\mathbb{A}}\left(\begin{bmatrix}
0 & X \\
Y & 0
\end{bmatrix}\right) &\leq \sqrt{\frac{1}{4}\max\Big\{{\big\|X^{\sharp_A}X + YY^{\sharp_A}\big\|}_{A}, {\big\|XX^{\sharp_A} + Y^{\sharp_A}Y\big\|}_{A}\Big\}
+ \frac{1}{2}\max\big\{w_{A}(XY), w_{A}(YX)\big\}}
\\& \leq \sqrt{\frac{1}{4}\max\Big\{{\|X^{\sharp_A}X\|}_{A} + {\|YY^{\sharp_A}\|}_{A}, {\|XX^{\sharp_A}\|}_{A} + {\|Y^{\sharp_A}Y\|}_{A}\Big\}
+ \frac{1}{2}\max\big\{{\|XY\|}_{A}, {\|YX\|}_{A}\big\}}
\\& \leq \sqrt{\frac{1}{4}\big ({\|X\|}^2_{A} + {\|Y\|}^2_{A}\big)
+ \frac{1}{2}{\|X\|}_{A}{\|Y\|}_{A}}
\\& = \frac{{\|X\|}_{A} + {\|Y\|}_{A}}{2},
\end{align*}
\endgroup
and hence
\begingroup\makeatletter\def\f@size{12}\check@mathfonts
\begin{align}\label{I.1.R.2.04}
w_{\mathbb{A}}\left(\begin{bmatrix}
0 & X \\
Y & 0
\end{bmatrix}\right) \leq \frac{{\|X\|}_{A} + {\|Y\|}_{A}}{2}.
\end{align}
\endgroup
On the other hand, since ${\begin{bmatrix}
0 & X \\
Y & 0
\end{bmatrix}}^2=\begin{bmatrix}
XY & 0 \\
0 & YX
\end{bmatrix}$, by the power inequality (\ref{I.1.04}) and Lemma \ref{L.2.01}(iii),
we have
\begingroup\makeatletter\def\f@size{10}\check@mathfonts
\begin{align*}
w^2_{\mathbb{A}}\left(\begin{bmatrix}
0 & X \\
Y & 0
\end{bmatrix}\right) \geq w_{\mathbb{A}}\left({\begin{bmatrix}
0 & X \\
Y & 0
\end{bmatrix}}^2\right)
= w_{\mathbb{A}}\left(\begin{bmatrix}
XY & 0 \\
0 & YX
\end{bmatrix}\right)
= \max\big\{w_{A}(XY), w_{A}(YX)\big\}
\end{align*}
\endgroup
and so
\begingroup\makeatletter\def\f@size{12}\check@mathfonts
\begin{align}\label{I.2.R.2.04}
\max\big\{w^{1/2}_{A}(XY), w^{1/2}_{A}(YX)\big\}
\leq w_{\mathbb{A}}\left(\begin{bmatrix}
0 & X \\
Y & 0
\end{bmatrix}\right).
\end{align}
\endgroup
Therefore, from (\ref{I.1.R.2.04}) and (\ref{I.2.R.2.04}) it follows that
\begingroup\makeatletter\def\f@size{12}\check@mathfonts
\begin{align*}
\max\big\{w^{1/2}_{A}(XY), w^{1/2}_{A}(YX)\big\}
\leq w_{\mathbb{A}}\left(\begin{bmatrix}
0 & X \\
Y & 0
\end{bmatrix}\right) \leq \frac{{\|X\|}_{A} + {\|Y\|}_{A}}{2}.
\end{align*}
\endgroup
\end{remark}
As an immediate consequence of the preceding theorem, we give an improvement of the second inequality in (\ref{I.1.02}).
\begin{corollary}\cite[Theorem 2.11]{Z.3}\label{C.2.05}
Let $X\in\mathbb{B}_{A^{1/2}}(\mathcal{H})$. Then
\begingroup\makeatletter\def\f@size{12}\check@mathfonts
\begin{align*}
w_{A}(X)\leq \frac{1}{2}\sqrt{{\big\|X^{\sharp_A}X + XX^{\sharp_A}\big\|}_{A} + 2w_{A}(X^2)}.
\end{align*}
\endgroup
\end{corollary}
\begin{proof}
By letting $Y = X$ in Theorem \ref{T.2.03}, we have
\begingroup\makeatletter\def\f@size{12}\check@mathfonts
\begin{align*}
w^2_{\mathbb{A}}\left(\begin{bmatrix}
0 & X \\
X & 0
\end{bmatrix}\right)\leq \frac{1}{4}{\big\|X^{\sharp_A}X + XX^{\sharp_A}\big\|}_{A} + \frac{1}{2}w_{A}(X^2).
\end{align*}
\endgroup
Now, by Lemma \ref{L.2.01} (iv), we deduce the desired result.
\end{proof}
The following results are another consequences of Theorem \ref{T.2.03} for certain $2\times 2$ operator matrices.
\begin{corollary}\label{C.2.06}
Let $X\in\mathbb{B}_{A^{1/2}}(\mathcal{H})$. Then
\begingroup\makeatletter\def\f@size{12}\check@mathfonts
\begin{align*}
w_{\mathbb{A}}\left(\begin{bmatrix}
0 & X \\
0 & 0
\end{bmatrix}\right) = w_{\mathbb{A}}\left(\begin{bmatrix}
0 & 0 \\
X & 0
\end{bmatrix}\right)=\frac{1}{2}{\|X\|}_{A}.
\end{align*}
\endgroup
\end{corollary}
\begin{proof}
By letting $Y = 0$ in Theorem \ref{T.2.03},
we have
\begingroup\makeatletter\def\f@size{12}\check@mathfonts
\begin{align*}
w^2_{\mathbb{A}}\left(\begin{bmatrix}
0 & X \\
0 & 0
\end{bmatrix}\right)\leq \frac{1}{4}{\|X^{\sharp_A}X \|}_{A},
\end{align*}
\endgroup
wherefrom, by (\ref{I.1.01}) we obtain
\begingroup\makeatletter\def\f@size{12}\check@mathfonts
\begin{align*}
w_{\mathbb{A}}\left(\begin{bmatrix}
0 & X \\
0 & 0
\end{bmatrix}\right)\leq \frac{1}{2}{\|X\|}_{A}.
\end{align*}
\endgroup
Also, by (\ref{I.1.02}) and Lemma \ref{L.2.01} (ii), we have
\begingroup\makeatletter\def\f@size{12}\check@mathfonts
\begin{align*}
\frac{1}{2}{\|X\|}_{A} = \frac{1}{2}{\left\|\begin{bmatrix}
0 & X \\
0 & 0
\end{bmatrix}\right\|}_{\mathbb{A}}
\leq w_{\mathbb{A}}\left(\begin{bmatrix}
0 & X \\
0 & 0
\end{bmatrix}\right).
\end{align*}
\endgroup
Thus $w_{\mathbb{A}}\left(\begin{bmatrix}
0 & X \\
0 & 0
\end{bmatrix}\right) = \frac{1}{2}{\|X\|}_{A}$.
\end{proof}
\begin{corollary}\label{C.2.07}
Let $T, S, X, Y\in\mathbb{B}_{A^{1/2}}(\mathcal{H})$. Then
\begin{itemize}
\item[(i)] $w_{\mathbb{A}}\left(\begin{bmatrix}
T & X \\
0 & 0
\end{bmatrix}\right) \leq w_{A}(T) +\frac{1}{2}{\|X\|}_{A}$.

\item[(ii)] $w_{\mathbb{A}}\left(\begin{bmatrix}
0 & 0 \\
Y & S
\end{bmatrix}\right)\leq w_{A}(S) +\frac{1}{2}{\|Y\|}_{A}$.
\end{itemize}
\end{corollary}
\begin{proof}
By the triangle inequality for $w_{\mathbb{A}}(\cdot)$, Lemma \ref{L.2.01} (iii) and Corollary \ref{C.2.06} it follows that
\begingroup\makeatletter\def\f@size{10}\check@mathfonts
\begin{align*}
w_{\mathbb{A}}\left(\begin{bmatrix}
T & X \\
0 & 0
\end{bmatrix}\right)& = w_{\mathbb{A}}\left(\begin{bmatrix}
T & 0 \\
0 & 0
\end{bmatrix}
+ \begin{bmatrix}
0 & X \\
0 & 0
\end{bmatrix}\right)
\\& \leq w_{\mathbb{A}}\left(\begin{bmatrix}
T & 0 \\
0 & 0
\end{bmatrix}\right) +
w_{\mathbb{A}}\left(\begin{bmatrix}
0 & X \\
0 & 0
\end{bmatrix}\right) = w_{A}(T) +\frac{1}{2}{\|X\|}_{A}.
\end{align*}
\endgroup
Hence $w_{\mathbb{A}}\left(\begin{bmatrix}
T & X \\
0 & 0
\end{bmatrix}\right) \leq w_{A}(T) +\frac{1}{2}{\|X\|}_{A}$.

By a similar argument, we get the inequality (ii).
\end{proof}
\begin{corollary}\label{C.2.08}
Let $T, S\in\mathbb{B}_{A^{1/2}}(\mathcal{H})$. Then
\begin{itemize}
\item[(i)] $w_{\mathbb{A}}\left(\begin{bmatrix}
T & S \\
T & S
\end{bmatrix}\right) \leq w_{A}(T-S) + \frac{1}{2}{\|T+S\|}_{A}$.

\item[(ii)] $w_{\mathbb{A}}\left(\begin{bmatrix}
T & S \\
-T & -S
\end{bmatrix}\right) \leq w_{A}(T+S) + \frac{1}{2}{\|T-S\|}_{A}$.
\end{itemize}
\end{corollary}
\begin{proof}
(i) Let $U = \frac{1}{\sqrt{2}}\begin{bmatrix}
I & I \\
-I & I
\end{bmatrix}$. Then $U$ is $\mathbb{A}$-unitary.
Consequently, by the identity (\ref{I.1.03}) and Corollary \ref{C.2.07}, we have
\begingroup\makeatletter\def\f@size{10}\check@mathfonts
\begin{align*}
w_{\mathbb{A}}\left(\begin{bmatrix}
T & S \\
T & S
\end{bmatrix}\right) &=
w_{\mathbb{A}}\left(U^{\sharp_A}\begin{bmatrix}
T & S \\
T & S
\end{bmatrix}U\right)
\\& = w_{\mathbb{A}}\left(\begin{bmatrix}
0 & 0 \\
T-S & T+S
\end{bmatrix}\right)
\leq w_{A}(T+S) + \frac{1}{2}{\|T-S\|}_{A},
\end{align*}
\endgroup
and hence
$w_{\mathbb{A}}\left(\begin{bmatrix}
T & S \\
T & S
\end{bmatrix}\right) \leq w_{A}(T-S) + \frac{1}{2}{\|T+S\|}_{A}$.

(ii) The proof is similar to (i) and so we omit it.
\end{proof}
\begin{corollary}\label{C.2.09}
Let $X\in\mathbb{B}_{A^{1/2}}(\mathcal{H})$. Then
\begin{itemize}
\item[(i)] $w_{\mathbb{A}}\left(\begin{bmatrix}
0 & X \\
X^{\sharp_A} & 0
\end{bmatrix}\right) = {\|X\|}_{A}$.
\item[(ii)] $\max\Big\{{\big\|X^{\sharp_A}X + X^{\sharp_A}(X^{\sharp_A})^{\sharp_A}\big\|}_{A},
{\big\|XX^{\sharp_A} + (X^{\sharp_A})^{\sharp_A}X^{\sharp_A}\big\|}_{A}\Big\} = 2{\|X\|}^2_{A}$.
\end{itemize}
\end{corollary}
\begin{proof}
First, note that $XX^{\sharp_A}$ and $X^{\sharp_A}X$ are $A$-selfadjoint, and then by (\ref{I.1.01}) we have
\begingroup\makeatletter\def\f@size{10}\check@mathfonts
\begin{align}\label{I.1.C.2.09}
w_{A}(XX^{\sharp_A}) = {\|XX^{\sharp_A}\|}_{A} = {\|X\|}^2_{A} = {\|X^{\sharp_A}X\|}_{A} = w_{A}(X^{\sharp_A}X).
\end{align}
\endgroup
Also, since
${\begin{bmatrix}
0 & X \\
X^{\sharp_A} & 0
\end{bmatrix}}^2 = \begin{bmatrix}
XX^{\sharp_A} & 0 \\
0 & X^{\sharp_A}X
\end{bmatrix}$, by Lemma \ref{L.2.01} (iii) and (\ref{I.1.C.2.09}) we obtain

\begingroup\makeatletter\def\f@size{10}\check@mathfonts
\begin{align*}
w_{\mathbb{A}}\left({\begin{bmatrix}
0 & X \\
X^{\sharp_A} & 0
\end{bmatrix}}^2\right) = \max\big\{w_{A}(XX^{\sharp_A}), w_{A}(X^{\sharp_A}X)\big\} = {\|X\|}^2_{A}.
\end{align*}
\endgroup
Further, by the power inequality (\ref{I.1.04}), Theorem \ref{T.2.03}, (\ref{I.1.01}) and (\ref{I.1.02})
we have
\begingroup\makeatletter\def\f@size{10}\check@mathfonts
\begin{align*}
{\|X\|}^2_{A} &= w_{\mathbb{A}}\left({\begin{bmatrix}
0 & X \\
X^{\sharp_A} & 0
\end{bmatrix}}^2\right)
\\& \leq
w^2_{\mathbb{A}}\left(\begin{bmatrix}
0 & X \\
X^{\sharp_A} & 0
\end{bmatrix}\right)
\\& \leq \frac{1}{4}\max\Big\{{\big\|X^{\sharp_A}X + X^{\sharp_A}(X^{\sharp_A})^{\sharp_A}\big\|}_{A}, {\big\|XX^{\sharp_A} + (X^{\sharp_A})^{\sharp_A}X^{\sharp_A}\big\|}_{A}\Big\}
\\& \qquad \qquad \qquad \qquad + \frac{1}{2}\max\big\{w_{A}(XX^{\sharp_A}), w_{A}(X^{\sharp_A}X)\big\}
\\& \leq \frac{1}{2}{\|X\|}^2_{A} + \frac{1}{2}{\|X\|}^2_{A} = {\|X\|}^2_{A}.
\end{align*}
\endgroup
Therefore, $w^2_{\mathbb{A}}\left(\begin{bmatrix}
0 & X \\
X^{\sharp_A} & 0
\end{bmatrix}\right) = {\|X\|}^2_{A}$ and
\begingroup\makeatletter\def\f@size{9}\check@mathfonts
\begin{align*}
\frac{1}{4}\max\Big\{{\big\|X^{\sharp_A}X + X^{\sharp_A}(X^{\sharp_A})^{\sharp_A}\big\|}_{A}, {\big\|XX^{\sharp_A} + (X^{\sharp_A})^{\sharp_A}X^{\sharp_A}\big\|}_{A}\Big\} = \frac{1}{2}{\|X\|}^2_{A}.
\end{align*}
\endgroup
Now, we deduce the desired results.
\end{proof}
In the following theorem, we establish an upper bound for the $\mathbb{A}$-numerical radius of
$2\times 2$ block matrices of semi-Hilbertian space operators.
\begin{theorem}\label{T.2.010}
Let $T, S, X, Y\in\mathbb{B}_{A^{1/2}}(\mathcal{H})$. Then

\begingroup\makeatletter\def\f@size{9}\check@mathfonts
\begin{align*}
w^2_{\mathbb{A}}\left(\begin{bmatrix}
T & X \\
Y & S
\end{bmatrix}\right) &\leq w^2_{\mathbb{A}}\left(\begin{bmatrix}
0 & X \\
Y & 0
\end{bmatrix}\right) + w_{\mathbb{A}}\left(\begin{bmatrix}
0 & XS \\
YT & 0
\end{bmatrix}\right)
\\& \qquad + \max\big\{w^2_{A}(T), w^2_{A}(S)\big\}
+ \frac{1}{2}\max\Big\{{\big\|T^{\sharp_A}T + XX^{\sharp_A}\big\|}_{A}, {\big\|S^{\sharp_A}S + YY^{\sharp_A}\big\|}_{A}\Big\}.
\end{align*}
\endgroup
\end{theorem}
\begin{proof}
We will assume that
$M=\begin{bmatrix}
0 & X \\
Y & 0
\end{bmatrix}$, $N=\begin{bmatrix}
XX^{\sharp_{A}} & 0 \\
0 & YY^{\sharp_{A}}
\end{bmatrix}$, $P=\begin{bmatrix}
T & 0 \\
0 & S
\end{bmatrix}$, $Q=\begin{bmatrix}
T^{\sharp_{A}}T & 0 \\
0 & S^{\sharp_{A}}S
\end{bmatrix}$ and $R=\begin{bmatrix}
0 & XS \\
YT & 0
\end{bmatrix}$, to simplify notations.

Therefore,
\begin{align}\label{I.1.T.2.010}
MP = R \quad \mbox{and} \quad Q+N = \begin{bmatrix}
T^{\sharp_{A}}T + XX^{\sharp_{A}} & 0 \\
0 & S^{\sharp_{A}}S + YY^{\sharp_{A}}
\end{bmatrix}.
\end{align}
Also, by Lemma \ref{L.2.01}(i), we obtain
\begin{align}\label{I.2.T.2.010}
MM^{\sharp_{\mathbb{A}}} = N \quad \mbox{and} \quad P^{\sharp_{\mathbb{A}}}P = Q.
\end{align}

Now, let $z\in \mathcal{H}\oplus\mathcal{H}$ with ${\|z\|}_{\mathbb{A}} = 1$.
Since $\begin{bmatrix}
T & X \\
Y & S
\end{bmatrix} = P+M$, we have

\begingroup\makeatletter\def\f@size{9}\check@mathfonts
\begin{align*}
\left|{\left\langle \begin{bmatrix}
T & X \\
Y & S
\end{bmatrix}z, z\right\rangle}_{\mathbb{A}}\right|^2
&= \Big|{\big\langle (P+M)z, z\big\rangle}_{\mathbb{A}}\Big|^2
\\& = \Big|{\langle Pz, z\rangle}_{\mathbb{A}}
+ {\langle Mz, z\rangle}_{\mathbb{A}}\Big|^2
\\& \leq \Big(\big|{\langle Pz, z\rangle}_{\mathbb{A}}\big|
+\big|{\langle Mz, z\rangle}_{\mathbb{A}}\big|\Big)^2
\\& = \big|{\langle Pz, z\rangle}_{\mathbb{A}}\big|^2
+ \big|{\langle Mz, z\rangle}_{\mathbb{A}}\big|^2 + 2\big|{\langle Pz, z\rangle}_{\mathbb{A}}\big|
\big|{\langle Mz, z\rangle}_{\mathbb{A}}\big|
\\& \leq w^2_{\mathbb{A}}(P)+
w^2_{\mathbb{A}}(M)
+ 2\big|{\langle Pz, z\rangle}_{\mathbb{A}}\big|
\big|{\langle z, M^{\sharp_{\mathbb{A}}}z\rangle}_{\mathbb{A}}\big|
\\& \leq w^2_{\mathbb{A}}(P)+ w^2_{\mathbb{A}}(M)
+ \big|{\langle Pz,M^{\sharp_{\mathbb{A}}}z\rangle}_{\mathbb{A}}\big|
+ {\|Pz\|}_{\mathbb{A}}{\|M^{\sharp_{\mathbb{A}}}z\|}_{\mathbb{A}}
\quad \big(\mbox{by Lemma \ref{L.2.02}}\big)
\\& = w^2_{\mathbb{A}}(P)+ w^2_{\mathbb{A}}(M)
+ \big|{\langle MPz, z\rangle}_{\mathbb{A}}\big|
+ \sqrt{{\langle Pz, Pz\rangle}_{\mathbb{A}}
{\langle M^{\sharp_{\mathbb{A}}}z, M^{\sharp_{\mathbb{A}}}z\rangle}_{\mathbb{A}}}
\\& = w^2_{\mathbb{A}}(P)+ w^2_{\mathbb{A}}(M)
+ \big|{\langle Rz, z\rangle}_{\mathbb{A}}\big|+\sqrt{{\langle P^{\sharp_{\mathbb{A}}}Pz, z\rangle}_{\mathbb{A}}
{\langle MM^{\sharp_{\mathbb{A}}}z, z\rangle}_{\mathbb{A}}}
\quad \big(\mbox{by (\ref{I.1.T.2.010})}\big)
\\& = w^2_{\mathbb{A}}(P)+ w^2_{\mathbb{A}}(M)
+ \big|{\langle Rz, z\rangle}_{\mathbb{A}}\big|+\sqrt{{\langle Qz, z\rangle}_{\mathbb{A}}
{\langle Nz, z\rangle}_{\mathbb{A}}}
\qquad \qquad \big(\mbox{by (\ref{I.2.T.2.010})}\big)
\\& \leq w^2_{\mathbb{A}}(P)+ w^2_{\mathbb{A}}(M)
+ w_{\mathbb{A}}(R) + \frac{1}{2}\Big({\langle Qz, z\rangle}_{\mathbb{A}} +
{\langle Nz, z\rangle}_{\mathbb{A}}\Big)
\\& \qquad \qquad \quad \quad \big(\mbox{by the arithmetic-geometric mean inequality}\big)
\\& = w^2_{\mathbb{A}}(P)+ w^2_{\mathbb{A}}(M)
+ w_{\mathbb{A}}(R) + \frac{1}{2}
{\big\langle(Q+N)z, z\big\rangle}_{\mathbb{A}}
\\& \leq w^2_{\mathbb{A}}(P)+ w^2_{\mathbb{A}}(M)
+ w_{\mathbb{A}}(R) + \frac{1}{2}{\|Q+N\|}_{\mathbb{A}}.
\\& \qquad \qquad \qquad \qquad \qquad \Big(\mbox{since $Q+N$ is an $\mathbb{A}$-positive operator}\Big)
\end{align*}
\endgroup
Thus
\begingroup\makeatletter\def\f@size{9}\check@mathfonts
\begin{align*}
w^2_{_{\mathbb{A}}}\left(\begin{bmatrix}
T & X \\
Y & S
\end{bmatrix}\right) &\leq w^2_{\mathbb{A}}(P)+ w^2_{\mathbb{A}}(M)
+ w_{\mathbb{A}}(R) + \frac{1}{2}{\|Q+N\|}_{\mathbb{A}}.
\end{align*}
\endgroup
Now, by (\ref{I.1.T.2.010}) and Lemma \ref{L.2.01}(ii)-(iii), we deduce the desired result.
\end{proof}
As an application of Theorem \ref{T.2.010}, we obtain the following result.
\begin{corollary}\label{C.2.011}
Let $T, X\in\mathbb{B}_{A^{1/2}}(\mathcal{H})$. Then

\begingroup\makeatletter\def\f@size{10}\check@mathfonts
\begin{align*}
\max\big\{w_{A}(T), w_{A}(X)\big\} &+ \frac{1}{2}\big|w_{A}(T+X) - w_{A}(T-X)\big|
\\&\leq \max\big\{w_{A}(T+X), w_{A}(T-X)\big\}
\\& \leq \sqrt{w^2_{A}(X) + w_{A}(XT) + w^2_{A}(T)
+ \frac{1}{2}{\big\|XX^{\sharp_A} + T^{\sharp_A}T\big\|}_{A}}.
\end{align*}
\endgroup
\end{corollary}
\begin{proof}
By letting $S = T$ and $Y = X$ in Theorem \ref{T.2.010}, and using Lemma \ref{L.2.01} (iv), we have
\begingroup\makeatletter\def\f@size{10}\check@mathfonts
\begin{align*}
w^2_{\mathbb{A}}\left(\begin{bmatrix}
T & X \\
X & T
\end{bmatrix}\right) &\leq w^2_{\mathbb{A}}\left(\begin{bmatrix}
0 & X \\
X & 0
\end{bmatrix}\right) + w_{\mathbb{A}}\left(\begin{bmatrix}
0 & XT \\
XT & 0
\end{bmatrix}\right)
\\& \qquad + \max\big\{w^2_{A}(T), w^2_{A}(T)\big\}
+ \frac{1}{2}\max\Big\{{\big\|T^{\sharp_A}T + XX^{\sharp_A}\big\|}_{A}, {\big\|T^{\sharp_A}T + XX^{\sharp_A}\big\|}_{A}\Big\}
\\& = w^2_{A}(X) + w_{A}(XT) + w^2_{A}(T)
+ \frac{1}{2}{\big\|T^{\sharp_A}T + XX^{\sharp_A}\big\|}_{A}.
\end{align*}
\endgroup
Hence
\begingroup\makeatletter\def\f@size{10}\check@mathfonts
\begin{align*}
w^2_{\mathbb{A}}\left(\begin{bmatrix}
T & X \\
X & T
\end{bmatrix}\right) \leq w^2_{A}(X) + w_{A}(XT) + w^2_{A}(T)
+ \frac{1}{2}{\big\|T^{\sharp_A}T + XX^{\sharp_A}\big\|}_{A}.
\end{align*}
\endgroup
Since by Lemma \ref{L.2.01} (iv) we have $w_{\mathbb{A}}\left(\begin{bmatrix}
T & X \\
X & T
\end{bmatrix}\right) = \max\big\{w_{A}(T+X), w_{A}(T-X)\big\}$, therefore by the above inequality we get
\begingroup\makeatletter\def\f@size{10}\check@mathfonts
\begin{align}\label{I.01.C.2.011}
\max\big\{w_{A}(T+X), w_{A}(T-X)\big\} \leq \sqrt{w^2_{A}(X) + w_{A}(XT) + w^2_{A}(T)
+ \frac{1}{2}{\big\|T^{\sharp_A}T + XX^{\sharp_A}\big\|}_{A}}.
\end{align}
\endgroup
On the other hand, by the triangle inequality for $w_{A}(\cdot)$, we have
\begingroup\makeatletter\def\f@size{10}\check@mathfonts
\begin{align*}
\max\big\{w_{A}(T+X), w_{A}(T-X)\big\} & = \frac{w_{A}(T+X) + w_{A}(T-X)}{2} + \frac{1}{2}\big|w_{A}(T+X) - w_{A}(T-X)\big|
\\& \geq \frac{w_{A}\big(T+X +T-X)}{2} + \frac{1}{2}\big|w_{A}(T+X) - w_{A}(T-X)\big|
\\& = w_{A}(T) + \frac{1}{2}\big|w_{A}(T+X) - w_{A}(T-X)\big|.
\end{align*}
\endgroup
Thus
\begingroup\makeatletter\def\f@size{10}\check@mathfonts
\begin{align}\label{I.02.C.2.011}
w_{A}(T) + \frac{1}{2}\big|w_{A}(T+X) - w_{A}(T-X)\big| \leq \max\big\{w_{A}(T+X), w_{A}(T-X)\big\}.
\end{align}
\endgroup
Similarly,
\begingroup\makeatletter\def\f@size{10}\check@mathfonts
\begin{align}\label{I.03.C.2.011}
w_{A}(X) + \frac{1}{2}\big|w_{A}(T+X) - w_{A}(T-X)\big| \leq \max\big\{w_{A}(T+X), w_{A}(T-X)\big\}.
\end{align}
\endgroup
So, by (\ref{I.02.C.2.011}) and (\ref{I.03.C.2.011}), we obtain
\begingroup\makeatletter\def\f@size{10}\check@mathfonts
\begin{align}\label{I.04.C.2.011}
\max\big\{w_{A}(T), w_{A}(X)\big\} + \frac{1}{2}\big|w_{A}(T+X) - w_{A}(T-X)\big|
\leq \max\big\{w_{A}(T+X), w_{A}(T-X)\big\}.
\end{align}
\endgroup
Finally, the result follows from the inequalities (\ref{I.01.C.2.011}) and (\ref{I.04.C.2.011}).
\end{proof}
\textbf{Acknowledgement.}
This work was supported a grant from the Science and Technology Commission of Shanghai Municipality (18590745200)
and the National Natural Science Foundation of China (11671261, 11971136).
\bibliographystyle{amsplain}

\end{document}